\documentclass{article}
\usepackage{upgreek}
\usepackage{latexsym,amssymb}
\usepackage{mathptmx,amsthm,mathrsfs}

\def\version{}

\usepackage{hyperref}
\hypersetup{
    colorlinks=true,
    linkcolor=blue,
    filecolor=magenta,      
    urlcolor=blue,
}

\newcommand\pa{\partial}

\newcommand\ov{\overline}

\newcommand\ve{\varepsilon}

\def\Re{{\rm Re\, }}
\def\Im{{\rm Im\,}}
\providecommand\C{{\mathbb C}}
\renewcommand\C{{\mathbb C}}
\newcommand{\R}{{\mathbb R}}

\newcommand{\const}{{\rm const}}

\font\thf cmssdc10 at 11pt

\theoremstyle{plain}
\newtheorem{theorem}{\thf Theorem}[section]
\newtheorem{lemma}[theorem]{\thf Lemma}
\newtheorem{corollary}[theorem]{\thf Corollary}
\newtheorem{proposition}[theorem]{\thf Proposition}
\theoremstyle{definition}

\theoremstyle{remark}
\newtheorem{remark}[theorem]{Remark}

\begin{document}

\title{Global well-posedness for Dirac equation with concentrated nonlinearity}

\author{
{\sc Elena Kopylova}
\footnote{Research supported by the Austrian Science Fund (FWF) under Grant No. P27492-N25 
and RFBR grant 18-01-00524.}
\\ 
{\it\small Faculty of Mathematics of Vienna University and IITP RAS
}}

\date{\version}

\maketitle

\begin{abstract}
We prove global well-posedness for  3D Dirac equation with a concentrated nonlinearity.
\end{abstract}
\section{Introduction}
\label{int-results}
We denote by $D_m$ the Dirac operator  $D_m:=-i\alpha\cdot\nabla+m\beta $,
where  $m>0$,  $\alpha_k$ with $k=1,2,3$ and $\beta$ are $4\times 4$ Dirac matrices.
We consider the Dirac field coupled to a nonlinear oscillator 
\begin{equation}\label{Dirac}
\left\{\begin{array}{c}
i\dot \psi(x,t)=D_m\psi(x,t)-D^{-1}_m\zeta(t)\delta(x)\\\\
\lim\limits_{\ve\to 0+}\lim\limits_{x\to 0}K_m^{\ve}\Big(\psi(x,t)-\zeta(t)g(x)\Big)= F(\zeta(t))
\end{array}\right|\quad x\in\R^3,\quad t\in\R.
\end{equation} 
Here $\psi(x,t)$, $\zeta(t)$  are vector functions with  values in $\C^4$,   $g(x)$ is the Green function of the operator $-\Delta+m^2$ in $\R^3$, 
\begin{equation}\label{Green}
  g(x)=\frac{e^{-m|x|}}{4\pi|x|},
\end{equation}
and $K_m^\ve=(-\Delta+m^2)^{-\ve}$ is a  smoothing operator, defined as
$$
(K_m^{\ve} \psi)(x)=\frac 1{(4\pi)^3}\int e^{-i\xi\cdot x}\frac{\hat \psi(\xi)\,d^3\xi}{(\xi^2+m^2)^\ve},\quad\ve\ge 0,
$$
where $\hat \psi(\xi)$ is the Fourier transform of $\psi(x)$.
Obviously, $(K_m^{\ve}\psi)(x)\to\psi(x)$  as $\ve\to0$. Hence, in the limit $\ve\to0$, the coupling in (\ref{Dirac}) formally depends on the value of the regular parts
$\psi_{reg}(x,t)=\psi(x,t)-\zeta(t)g(x)$ of the Dirac field $\psi(x,t)$  at one point $x=0$ only.

We assume that the nonlinearity $F_j(\zeta)=F_j(\zeta_j)$
admits a real-valued potential:
\begin{equation}\label{FU}
 F_j(\zeta_j)=\pa_{\ov\zeta_j} U(\zeta),\quad U\in C^2(\C^4),\quad j=1,....,4,
\end{equation} 
 where  $\pa_{\ov\zeta_j}:=\frac12(\frac{\pa U}{\pa\zeta_{j1}}+i\frac{\pa U}{\pa\zeta_{j2}})$ with $\zeta_{j1}:=\Re\zeta_j$ and $\zeta_{j2}:=\Im\zeta_j$, and
\begin{equation}\label{bound-below}
U(\zeta)\ge b|\zeta|^2-a, \quad{\rm for}\ \zeta\in\C^4,\quad
{\rm where}\ b> 0~~~{\rm and}~~~a\in\R.
\end{equation}
Our main result is as follows. For  initial data of type
\begin{equation}\label{in_d}
\psi(x,0)=f(x)+\zeta_{0}g(x), \quad f\in H^2(\R^3)\otimes\C^4, \quad\zeta_0\in\C^4,
\end{equation}
we prove a global well-posedness of the Cauchy problem for  the system(\ref{Dirac}).

Let us comment on our approach. 
We develop the approach which was  introduced in \cite{K16, NP} in the context of the Klein-Gordon  and wave equations.
First, we obtain  some regularity  properties  i) of   solutions  $\psi_{\rm free}(x,t)$ to the  free Dirac equation 
with  initial data (\ref{in_d}),  and ii) of  solutions $\psi_{S}(x,t)$  to the   Dirac equation   with zero initial data and with the source $D_m^{-1}\xi(t)\delta(x)$, 
where $\xi \in C^1 [0,\infty)$ ( Lemmas \ref{imp} and \ref{lim_lam}, and Propositions \ref{Phi_lim}  and \ref{PropD}). 
We use these regularity properties to prove the existence of a local solution to (\ref{Dirac}) of the type
\[
\psi(x,t)=\psi_{\rm free}(x,t)+\psi_{S}(x,t).
\]
We show that  $\zeta(t)$ is a solution to a  first-order nonlinear integro-differential equation driven by $\psi_{\rm free}(0,t)$. 
Then we prove that  conditions  (\ref{FU})--(\ref{bound-below}) provide the energy conservation.  
Finally, we use the energy conservation to obtain the global existence theorem.  
Let us note  that the system (\ref{Dirac}) without smoothing operator $K_m^{\ve}$  is not well posed (see Remark \ref{rek}).

As was noted above, the Dirac equation with concentrated nonlinearities  is not well posed
in contrast to corresponding Klein-Gordon  equation \cite{K16} and wave equation \cite{NP}.  
So  we should to introduce a smoothing operator $K_m^{\ve}$ resembling  the Pauli-Willars  renormalization.
As the result, we have found a novel model of nonlinear point interaction which provides the Hamilton structure and needed a priori estimates.
However, the introduction  of the smoothing operator in (\ref{Dirac}) leads to additional  difficulties in  justification of numerous   limits.
We overcome these difficulties  using  subtle properties of special functions. 
\section{Main result}
We denote by $a-$ any number $a-\ve$ with an arbitrary small, but fixed $\ve > 0$. We  fix a nonlinear function $F:\C^4\to\C^4$ and define the domain
\begin{eqnarray}\nonumber
{\cal D}_F&=&\{\psi\in L^2(\R^3)\otimes\C^4: \psi(x)=\psi_{reg}(x)+\zeta g(x), ~~ \zeta\in\C^4,~~\psi_{reg}\in H^{\frac 32-}(\R^3)\otimes\C^4,\\
\nonumber
&&\exists\lim_{\ve\to 0+}\lim_{x\to 0} K_m^\ve\psi_{reg}(x)=F(\zeta)\},
\end{eqnarray}
which generally is not a linear space. 
Everywhere below we will write $L^2(\R^3)$ and $H^s(\R^3)$ instead of  $L^2(\R^3)\otimes\C^4$ and   $H^{s}(\R^3)\otimes\C^4$.
Denote $\Vert\cdot\Vert=\Vert\cdot\Vert_{L^2(\R^3)}$.
\begin{theorem}\label{theorem-well-posedness}
Let  conditions (\ref{FU}) and (\ref{bound-below}) hold. Then 
\begin{enumerate}
\item
For every initial data $\psi_0(x)=f(x)+\zeta_0  g(x)$ with $f\in H^{2}(\R^3)$  the equation (\ref{Dirac})
has a unique solution $\psi(x,t)$ such that 
\[
\psi(\cdot,t)\in C(\R,{\cal D}_F).
\]
\item
The following conservation law holds:
\begin{equation}\label{ec}
{\cal H}_F(\psi(\cdot,t)):=
\Vert D_m\psi_{reg}(\cdot,t)\Vert^2+U(\zeta(t))=\const, \quad t\in\R.
\end{equation}
\item
The following a priori bound holds:
\begin{equation}\label{apb}
|\zeta(t)|\le C(\psi_0),\quad t\in\R. 
\end{equation}
\end{enumerate}
\end{theorem}
Obviously, it suffices to prove Theorem \ref {theorem-well-posedness} for $t\ge 0$. We will do it in Section \ref{nonlin-sect}.
Previously, we obtain some necessary properties of the free  Dirac equation an of the Dirac equations  with   sources of special kinds.
\section{Free Dirac equation}
\label{fD-sect}
Consider   the solution $\psi_f(x,t)$  to the free Dirac equation
\begin{equation}\label{CP1}
i\dot \psi_f(x,t)=D_m \psi_f(x,t),\quad \psi_f(x,0)=f(x),
\end{equation}
with  initial data $f\in H^2(\R^3)$. Evidently,  $\psi_f(\cdot,t)\in C([0,\infty), H^2(\R^3))$.
Denote
\begin{equation}\label{lam}
\lambda(t):=\psi_f(0,t)\in C[0,\infty)\otimes\C^4.
\end{equation}
\begin{lemma}\label{imp}
\begin{equation}\label{dot-lam}
\dot\lambda\in L^2_{loc}[0,\infty)\otimes\C^4.
\end{equation}
\end{lemma}
\begin{proof}
We represent the solution $\psi_f(x,t)$ to  (\ref{CP1}) as 
\begin{equation}\label{ch-u-s}
\psi_f(x,t)=(-i\partial_t-D_m)u(x,t),
\end{equation} 
where vector function $u(x,t)$ is a solution to the free Klein-Gordon equation 
$$
\ddot u(x,t)=(\Delta-m^2)u(x,t), \quad u(x,0)=-D_m^{-1}f(x),\quad \dot u(x,0)=0.
$$
Then the  functions $v(x,t)=\dot u(x,t)$ and $w(x,t)=D_m u(x,t)$ satisfy
$$
\ddot v(x,t)=(\Delta-m^2)v(x,t), \quad v(x,0)=0,\quad \dot v(x,0)=D_m f(x),
$$
$$
\ddot w(x,t)=(\Delta-m^2)w(x,t), \quad w(x,0)=-f(x),\quad \dot w(x,0)=0.
$$
It is obvious that $(v(x,0),\dot v(x,0)), (w(x,0),\dot w(x,0))\in H^2(\R^3)\oplus H^1(\R^3)$. Then
applying \cite[Corollary 4.3]{K16}, we obtain
$$
\dot v(0,t),\dot w(0,t)\in L^2_{loc}[0,\infty)\otimes\C^4.
$$
Hence, (\ref{ch-u-s}) implies 
$$
\dot \psi_f(0,t)=-i\ddot u(0,t)-D_m\dot u(0,t)=-i\dot v(0,t)-\dot w(0,t)\in L^2_{loc}[0,\infty)\otimes\C^4.
$$
\end{proof}
Now we consider the free Dirac equation with initial data $eg(x)$ with arbitrary  $e\in \C^4$,
\begin{equation}\label{CP2}
i\dot \psi_{eg}(x,t)=D_m \psi_{eg}(x,t),\quad \psi_{eg}(x,0)=eg(x),
\end{equation}
 and obtain  explicit   formula  for the solution  $\psi_{eg}(x,t)$.  Note that the function
\begin{equation}\label{fi}
\phi_{eg}(x,t):=\psi_{eg}(x,t)-eg(x)
\end{equation}
satisfies 
\[
i\dot\phi_{eg}(x,t)=D_m \phi_{eg}(x,t)+D_m^{-1}e\delta(x),\quad \phi_{eg}(x,0)=0,
\]
since $D_m eg(x)=D_m^{-1}D_m^2 eg=D_m^{-1}e\delta(x)$. 
Similarly to (\ref{ch-u-s}), we  represent  $\phi_{eg}(x,t)$ as
\begin{equation}\label{chw}
\phi_{eg}(x,t)= (-i\partial_t-D_m)D_m^{-1}e\gamma(x,t)=-iD_m^{-1}e\dot\gamma(x,t)-e\gamma(x,t),
\end{equation}
where
\begin{equation}\label{gamma}
\gamma(x,t)=\frac{\theta(t-|x|)}{4\pi|x|}
-\frac{m}{4\pi}\int_0^t\frac{\theta(s-|x|)J_1(m\sqrt{s^2-|x|^2})}
{\sqrt{s^2-|x|^2}}ds
\end{equation}
is the solution to the Klein-Gordon equation
\begin{equation}\label{gamma-eq}
\ddot\gamma(x,t)=(\Delta-m^2)\gamma(x,t)+\delta(x),\quad \gamma(x,0)=0,\quad \dot\gamma(x,0)=0.
\end{equation}
Here $\theta$ is the Heaviside function and $J_1$ is the Bessel function of the first order.
Finally,  (\ref{fi}) and (\ref{chw}) imply
\begin{equation}\label{rw}
\psi_{eg}(x,t)=eg(x)-e\gamma(x,t)-iD_m^{-1}e\dot\gamma(x,t).
\end{equation}
\begin{lemma}\label{lim_lam}
For any $t>0$ there exists 
\begin{equation}\label{mu}
\mu(t):=\lim_{\ve \to 0+}\lim_{x \to 0} K_m^{\ve}(g(x)-\gamma(x,t))=\lim_{x \to 0}(g(x)-\gamma(x,t))
=-\frac{m}{4\pi}+\frac{m}{4\pi}\int_0^t\frac{J_1(ms)}{s}ds.
\end{equation}
The function $\mu(t)$ is continuous  for $t>0$, and there exists 
$$
\mu(0)=\lim_{t \to 0}\mu(t)=-\frac{m}{4\pi}.
$$
\end{lemma}
\begin{proof}
Applying the Fourier transform $\hat f(\xi)={\cal F}_{x\to \xi} f(x)$, we get  
\begin{eqnarray}\nonumber
\hat g(\xi,t)-\hat\gamma(\xi,t)&=&\frac{1}{\xi^2+m^2}-\int_0^t\frac{\sin s\sqrt{\xi^2+m^2}}{\sqrt{\xi^2+m^2}} ds
=\frac{\cos t\sqrt{\xi^2+m^2}}{\xi^2+m^2} \\
\label{np01}
&=&\frac{\cos t\sqrt{\xi^2+m^2}-\cos |\xi|t}{\xi^2}-\frac{m^2\cos t\sqrt{\xi^2+m^2}}{\xi^2(\xi^2+m^2)}+\frac {\cos |\xi|t}{\xi^2},\quad t>0.
\end{eqnarray}
Then  for (\ref{mu}) it suffices to justify the following permutation of the  limits:
\begin{eqnarray}\nonumber
&&\lim_{\ve \to 0+}\lim_{x \to 0}K_m^{\ve}{\cal F}^{-1}_{\xi\to x}
\Big(\frac{\cos t\sqrt{\xi^2+m^2}-\cos |\xi|t}{\xi^2}-\frac{m^2\cos t\sqrt{\xi^2+m^2}}{\xi^2(\xi^2+m^2)}+\frac {\cos |\xi|t}{\xi^2}\Big)\\
\label{np1}
&&=\lim_{x \to 0}{\cal F}^{-1}_{\xi\to x}\Big(\frac{\cos t\sqrt{\xi^2+m^2}-\cos |\xi|t}{\xi^2}
-\frac{m^2\cos t\sqrt{\xi^2+m^2}}{\xi^2(\xi^2+m^2)}+\frac {\cos |\xi|t}{\xi^2}\Big),\quad t>0.
\end{eqnarray}
We will do it  for each term in (\ref{np1}) separately.\\
{\it Step i)} It is obvious that
\begin{eqnarray}\nonumber
\lim_{\ve\to 0+}\lim_{x\to 0}K_m^{\ve}{\cal F}^{-1}_{\xi\to x}\frac{m^2\cos t\sqrt{\xi^2+m^2}}{\xi^2(\xi^2+m^2)}
=\lim_{x\to 0} {\cal F}^{-1}_{\xi\to x}\frac{m^2\cos t\sqrt{\xi^2+m^2}}{\xi^2(\xi^2+m^2)}=\frac{m^2}{2\pi^2}\int_0^\infty \frac{\cos t\sqrt{r^2+m^2}}{r^2+m^2}dr.
\end{eqnarray}
{\it Step ii)} Now we prove that
$$
\lim_{\ve \to 0+}\lim_{x \to 0} K_m^{\ve}{\cal F}^{-1}_{\xi\to x}\frac {\cos |\xi|t}{\xi^2}=\lim_{x \to 0}{\cal F}^{-1}_{\xi\to x}\frac {\cos |\xi|t}{\xi^2}=0,\quad t>0.
$$
First, note that
$$
\lim_{x \to 0}{\cal F}^{-1}_{\xi\to x}\frac {\cos |\xi|t}{\xi^2}
=\lim_{\rho \to 0+}\frac{1}{2\pi^2}\int_0^\infty\frac{\sin r\rho\cos tr}{r\rho}dr
=\lim_{ \rho\to 0+}\frac{1}{4\pi^2\rho}\int_0^\infty\frac{\sin r(t+\rho)-\sin r( t-\rho)}{r}dr=0,\quad \rho=|x|,
$$
since
$$
\int_0^\infty\frac{\sin r\alpha}{r}dr=\int_0^\infty\frac{\sin u}{u}du,\quad\alpha>0.
$$
On the other hand, 
\begin{eqnarray}\nonumber
\!\!\!\!\!\!\!\!\!\!\!\!\!\!\! \!\!\!\!\! \!\!\!\!\! \!\!\!\!\! 
 \lim_{x \to 0} K_m^{\ve}{\cal F}^{-1}_{\xi\to x}\frac {\cos |\xi|t}{\xi^2}&=&\lim_{x\to 0}{\cal F}^{-1}_{\xi\to x}\frac {\cos |\xi|t}{\xi^2(\xi^2+m^2)^{\ve}}
=\lim_{\rho\to 0+}\frac{1}{2\pi^2\rho}\int\limits_0^\infty\frac{\sin r\rho\cos tr}{r(r^2+m^2)^\ve}dr\\
\nonumber 
&=&\frac{1}{2\pi^2}\int\limits_0^\infty\frac{\cos tr\,dr}{(r^2+m^2)^\ve}=\frac{1}{2\pi^2}\Big(\frac{2m}t\Big)^{\frac 12-\ve}\frac{\sqrt{\pi}{\bf K}_{\ve-\frac 12}(mt)} {\Gamma(\ve)}
\end{eqnarray}
by \cite [Formula  1.3.(7)]{E}. Here ${\bf K}_\nu$  is the modified Bessel function, and $\Gamma$ is the gamma function.
One can justify the last limit, splitting the integral into a sum of  integrals over the intervals
$[0,1]$ and  $[1,\infty)$, and integrating by parts in the second one.
Therefore,
\begin{eqnarray}\nonumber
\lim_{\ve \to 0+}\lim_{x \to 0} K_m^{\ve}{\cal F}^{-1}_{\xi\to x}\frac {\cos |\xi|t}{\xi^2}
=\lim_{\ve\to 0+}\frac{1}{2\pi^2}\Big(\frac{2m}t\Big)^{\frac 12-\ve}\frac{\sqrt{\pi}{\bf K}_{\ve-\frac 12}(mt)} {\Gamma(\ve)}=0,\quad t>0,
\end{eqnarray}
since  
$$
\lim_{\ve\to 0+}\frac{1} {\Gamma(\ve)}=0,\qquad
\lim_{\ve\to 0+}{\bf K}_{\ve-\frac 12}(mt)= {\bf K}_{-\frac 12}(mt)=\Big(\frac{\pi}{2mt}\Big)^{\frac 12}e^{-mt}
$$
by \cite[Formulas 5.7.1 and   10.39.2]{O}. \\
{\it Step iii)} It remains to check that
\begin{equation}\label{npf}
\lim_{\ve\to 0+}\lim_{x\to 0} K_m^{\ve}{\cal F}^{-1}_{\xi\to x}\frac{\cos t\sqrt{\xi^2+m^2}-\cos t|\xi|}{\xi^2}
=\lim_{x\to 0}{\cal F}^{-1}_{\xi\to x}\frac{\cos t\sqrt{\xi^2+m^2}-\cos t|\xi|}{\xi^2},\quad t>0.
\end{equation}
One has
$$
K_m^{\ve}{\cal F}^{-1}_{\xi\to x}\frac{\cos t\sqrt{\xi^2\!+m^2}\!-\!\cos t|\xi|}{\xi^2}
=\frac{1}{2\pi^2}\!\int\limits_{0}^{2m} \frac{(\cos t\sqrt{r^2\!+m^2}\!-\cos tr)\sin\rho r}{\rho r(r^2+m^2)^{\ve}}dr
+\sum_{\pm}\frac{1}{4\pi^2}\!\int\limits_{2m}^{\infty}\frac{(e^{\pm it\sqrt{r^2+m^2}}\!-e^{\pm itr})\sin\rho r}{\rho r(r^2+m^2)^{\ve}}dr,
$$
where $\rho=|x|$.  Evidently,
$$
\lim_{\ve\to 0+}\lim_{\rho\to 0+} \int\limits_{0}^{2m} \frac{(\cos t\sqrt{r^2\!+m^2}\!-\cos tr)\sin\rho r}{\rho r(r^2+m^2)^{\ve}}dr
=\lim_{\rho\to 0+} \int\limits_{0}^{2m} \frac{(\cos t\sqrt{r^2\!+m^2}\!-\cos tr)\sin\rho r}{\rho r}dr.
$$
Hence, (\ref{npf}) will follow from
\begin{equation}\label{npf+}
\lim_{\ve\to 0+}\lim_{\rho\to 0+}\int\limits_{2m}^{\infty}\frac{(e^{\pm it\sqrt{r^2+m^2}}\!-e^{\pm itr})\sin\rho r}{\rho r(r^2+m^2)^{\ve}}dr
=\lim_{\rho\to 0+}\int\limits_{2m}^{\infty}\frac{(e^{\pm it\sqrt{r^2+m^2}}\!-e^{\pm itr})\sin\rho r}{\rho r}dr.
\end{equation}
Note, that
$$
e^{\pm it\sqrt{r^2+m^2}}-e^{\pm itr}=e^{\pm itr}\Big(\pm\frac{itm^2}{2r}+R_{\pm}(r,t)\Big),\quad r\ge 2m,\quad t>0,
$$
where
$$
|R_{\pm}(r,t)|\le C(m)(1+t)^2/r^2,\qquad |r|\ge 2m.
$$
The last estimate implies 
\begin{equation}\label{I0}
\lim_{\ve\to 0+}\lim_{\rho\to 0+}\int\limits_{2m}^{\infty}\frac{e^{\pm itr}R_{\pm}(r,t)\sin\rho r}{\rho r(r^2+m^2)^{\ve}}dr
=\lim_{\rho\to 0+}\int\limits_{2m}^{\infty}\frac{e^{\pm itr}R_{\pm}(r,t)\sin\rho r}{\rho r}dr=\int\limits_{2m}^{\infty}e^{\pm itr}R_{\pm}(r,t)\,dr.
\end{equation}
Moreover,
\begin{equation}\label{I2}
\lim_{\ve\to 0+}\lim_{\rho\to 0+}\int\limits_{2m}^\infty \frac{e^{\pm itr}\sin r\rho\,dr}{\rho r^2(r^2+m^2)^{\ve}}
=\lim_{\rho\to 0+}\int\limits_{2m}^\infty \frac{e^{\pm itr}\sin r\rho\,dr}{\rho r^2}=\int\limits_{2m}^\infty \frac{e^{\pm itr}\,dr}{r},\quad t>0,
\end{equation}
that easily follows by means of integration by parts.  Finally,  (\ref{I0}) --(\ref{I2}) imply (\ref{npf+}).
\end{proof}
\section{Linear Dirac equation with  sources}
\label{dD-sect}
\subsection {Dirac equation with the  source $D_m^{-1}\zeta(t)\delta(x)$}
For arbitrary  $\zeta(t)\in C^1 [0,\infty)\otimes\C^4$,   consider  the  equation 
\begin{equation}\label{CP3}
i\dot\psi_S(x,t)= D_m \psi_S(x,t) -D^{-1}_m\zeta(t)\delta(x), \quad \psi_S(x,0) = 0.
\end{equation} 
It is easy to verify that 
$$
\psi_S(x,t):=(i\pa_t+D_m)D_m^{-1}\varphi_S(x,t),
$$
where 
\begin{equation}\label{vp1}
\varphi_S(x,t):= \frac{\theta(t-|x|)}{4\pi|x|}\zeta(t-|x|)-\frac{m}{4\pi}
\int_0^t\frac{\theta(s-|x|)J_1(m\sqrt{s^2-|x|^2})}{\sqrt{s^2-|x|^2}}\zeta(t-s)ds
\end{equation}
 is the solution to the  Klein-Gordon equation with $\delta$-like source:
 $$
\ddot\varphi_S(x,t)=(\Delta-m^2)\varphi_S(x,t)+\zeta(t)\delta(x), \quad \varphi_S(x,0) = 0,\quad\dot\varphi_S(x,0)=0.
$$
Hence
\begin{equation}\label{vp}
\psi_S(x,t)= \varphi_S(x,t)+iD_m^{-1}\zeta_0\dot\gamma(x,t)+ iD_m^{-1}p_S(x,t),
\end{equation}
where $\gamma(x,t)$ is defined in (\ref{gamma}), and
\begin{equation}\label{p}
p_S(x,t):=\frac{\theta(t-|x|)\dot\zeta(t-|x|)}{4\pi|x|}
-\frac{m}{4\pi}\int_0^t\frac{\theta(s-|x|)J_1(m\sqrt{s^2-|x|^2})}{\sqrt{s^2-|x|^2}}\dot\zeta(t-s)ds.
\end{equation}
\begin{lemma}\label{Phi_lim}
For any  $\zeta(t) \in C^1 [0,\infty)\otimes\C^4$ such that $\ddot\zeta(t)\in L^2_{loc}[0,\infty)\otimes\C^4$, one has
\begin{eqnarray}\nonumber
\lim_{\ve\to 0+}\lim_{x\to 0}K_m^{\ve}\left(\varphi_S(x,t)-\zeta(t)g(x)\right)&=&\lim_{x\to 0} (\varphi_S(x,t)-\zeta(t)g(x))\\
\label{dif1}
&=&\frac{1}{4\pi}\Big(m\zeta(t)-\dot\zeta(t)-m\int_0^t\frac{J_1(ms)}{s}\zeta(t-s)ds\Big),\quad t>0.
\end{eqnarray}
\end{lemma}
\begin{proof}
The Fourier transform of $\varphi_S(x,t)-\zeta(t)g(x)$ for any $t>0$ reads
$$
\hat \varphi_S(k,t)-\zeta(t)\hat g(k)=\int_0^t\frac{\sin s\sqrt{\xi^2+m^2}}{\sqrt{\xi^2+m^2}}\zeta(t-s)ds-\frac{\zeta(t)}{\xi^2+m^2}
=-\frac{\cos t\sqrt{\xi^2+m^2}}{\xi^2+m^2}\zeta(0)-\int_0^t \frac{\cos s\sqrt{\xi^2+m^2}}{\xi^2+m^2}\dot\zeta(t-s)ds.
$$
Due to (\ref{np01}) - (\ref{np1}),
$$
\lim_{\ve \to 0+}\lim_{x \to 0}K_m^{\ve}{\cal F}^{-1}_{\xi\to x}\frac{\cos t\sqrt{\xi^2+m^2}}{\xi^2+m^2}
=\lim_{x \to 0}{\cal F}^{-1}_{\xi\to x} \frac{\cos t\sqrt{\xi^2+m^2}}{\xi^2+m^2}.
$$
Hence, it remains to prove that
\begin{equation}\label{np9}
\lim_{\ve \to 0+}\lim_{x \to 0}K_m^{\ve}{\cal F}^{-1}_{\xi\to x}\int_0^t \frac{\cos s\sqrt{\xi^2+m^2}}{\xi^2+m^2}\dot\zeta(t-s)ds
=\lim_{x \to 0}{\cal F}^{-1}_{\xi\to x} \int_0^t \frac{\cos s\sqrt{\xi^2+m^2}}{\xi^2+m^2}\dot\zeta(t-s)ds.
\end{equation}
Integrating by parts, we obtain
$$
 \int_0^t \frac{\cos s\sqrt{\xi^2+m^2}}{\xi^2+m^2}\dot\zeta(t-s)ds=\frac {\sin t\sqrt{\xi^2+m^2}}{(\xi^2+m^2)\sqrt{\xi^2+m^2}}\dot\zeta(0)
 + \int_0^t \frac{\sin s\sqrt{\xi^2+m^2}}{(\xi^2+m^2)\sqrt{\xi^2+m^2}}\ddot\zeta(t-s)ds.
 $$
Note,  that
 $$
 \lim_{\ve \to 0+}\lim_{x \to 0}K_m^{\ve}{\cal F}^{-1}_{\xi\to x}\frac{\sin t\sqrt{\xi^2+m^2}}{(\xi^2+m^2)\sqrt{\xi^2+m^2}} 
 =\lim_{x \to 0}{\cal F}^{-1}_{\xi\to x}\frac{\sin t\sqrt{\xi^2+m^2}}{(\xi^2+m^2)\sqrt{\xi^2+m^2}},\quad t>0,
 $$
 or equivalently,
 $$
  \lim_{\ve \to 0+}\lim_{\rho \to 0+}\frac{1}{2\pi^2}\int\limits_0^\infty\frac{\sin r\rho\sin t\sqrt{r^2+m^2}\, r dr}{\rho(r^2+m^2)^{1+\ve}\sqrt{r^2+m^2}}
 =\lim_{\rho \to 0+}\frac{1}{2\pi^2}\int\limits_0^\infty\frac{\sin r\rho\sin t\sqrt{r^2+m^2}\, r dr}{\rho(r^2+m^2)\sqrt{r^2+m^2}},\quad\rho=|x|,\quad t>0.
 $$
 One can easily justify this  by partial integration.
Hence,  for (\ref{np9}), it remains to prove that
\begin{equation}\label{use}
 \lim_{\ve \to 0+}\lim_{x \to 0}K_m^{\ve}{\cal F}^{-1}_{\xi\to x}\int_0^t \frac{\sin s\sqrt{\xi^2+m^2}}{(\xi^2+m^2)\sqrt{\xi^2+m^2}}\ddot\zeta(t-s)ds
  =\lim_{x \to 0}{\cal F}^{-1}_{\xi\to x}\int_0^t \frac{\sin s\sqrt{\xi^2+m^2}}{(\xi^2+m^2)\sqrt{\xi^2+m^2}}\ddot\zeta(t-s)ds,\quad t>0.
\end{equation}
We split the integrand in (\ref{use}) as
\begin{equation}\label{np10}
\frac{\sin s\sqrt{\xi^2\!+m^2}}{(\xi^2\!+m^2)\sqrt{\xi^2\!+m^2}} \ddot\zeta(t\!-s)
=\Big(\frac{\sin s\sqrt{\xi^2\!+m^2}-\sin s|\xi|}{\xi^2\sqrt{\xi^2\!+m^2}}-\frac{m^2\sin s\sqrt{\xi^2\!+m^2}}{\xi^2(\xi^2\!+m^2)\sqrt{\xi^2\!+m^2}}
+\frac {\sin s|\xi|}{\xi^2\sqrt{\xi^2\!+m^2}}\Big)\ddot\zeta(t-\!s),
\end{equation}
and justify the permutation of the  limits (\ref{use}) for integrals of each terms in the RHS of (\ref{np10}) separately.
\medskip\\
{\it Step i)}
First, consider the integral of the second term of (\ref{np10}). By the  Fubini  theorem 
\begin{eqnarray}\nonumber
K_m^{\ve}{\cal F}^{-1}_{\xi\to x}\int\limits_0^t\frac{m^2\sin s\sqrt{\xi^2+m^2}}{\xi^2(\xi^2+m^2)\sqrt{\xi^2+m^2}}\ddot\zeta(t-s)ds
&=&\frac{m^2}{(2\pi)^3}\int\limits_{\R^3} e^{-i\xi\cdot x}\Big(\int\limits_0^t\frac{\sin s\sqrt{\xi^2+m^2}}{\xi^2(\xi^2+m^2)^{1+\ve}\sqrt{\xi^2+m^2}}\ddot\zeta(t-s)ds\Big)d^3\xi\\
\nonumber
&=&\frac{m^2}{(2\pi)^3}\int\limits_0^t \Big(\int\limits_{\R^3}\frac{e^{-i\xi\cdot x}\sin s\sqrt{\xi^2+m^2}}{\xi^2(\xi^2+m^2)^{1+\ve}\sqrt{\xi^2+m^2}}d^3\xi\Big)\ddot\zeta(t-s)ds,\quad\ve\ge 0. 
\end{eqnarray}
Hence,
$$
\lim_{\ve\to 0+}\lim_{x\to 0}K_m^{\ve}{\cal F}^{-1}_{\xi\to x}\int\limits_0^t\frac{m^2\sin s\sqrt{\xi^2+m^2}}{\xi^2(\xi^2+m^2)\sqrt{\xi^2+m^2}}\ddot\zeta(t-s)ds
=\lim_{x \to 0}{\cal F}^{-1}_{\xi\to x}\int_0^t\frac{m^2\sin s\sqrt{\xi^2+m^2}}{\xi^2(\xi^2+m^2)\sqrt{\xi^2+m^2}}\ddot\zeta(t-s)ds 
$$
by the  Lebesgue theorem.\\
{\it Step ii)} Similarly, for the integral of the first term, we obtain
$$
\lim_{\ve\to 0+}\lim_{x\to 0}K_m^{\ve}{\cal F}^{-1}_{\xi\to x}\int_0^t\frac{\sin s\sqrt{\xi^2+m^2}-\sin s|\xi|}{\xi^2\sqrt{\xi^2+m^2}}\ddot\zeta(t-s)ds
=\lim_{x \to 0}{\cal F}^{-1}_{\xi\to x}\int_0^t \frac{\sin s\sqrt{\xi^2+m^2}-\sin s|\xi|}{\xi^2\sqrt{\xi^2+m^2}}\ddot\zeta(t-s)ds 
$$
since 
$$
\sin t\sqrt{\xi^2+m^2}-\sin t|\xi|=\frac{1}{2i}[(e^{it\sqrt{\xi^2+m^2}}-e^{it|\xi|})-(e^{-it\sqrt{\xi^2+m^2}}-e^{-it|\xi|})]\sim t|\xi|^{-1},\quad |\xi|\to\infty.
$$
{\it Step iii)} It remains to consider the third term of (\ref{np10}) and prove that
\begin{equation}\label{use1}
\lim_{\ve\to 0+}\lim_{x\to 0}K_m^{\ve}{\cal F}^{-1}_{\xi\to x}\int_0^t\frac {\sin s|\xi|}{\xi^2\sqrt{\xi^2+m^2}}\ddot\zeta(t-s)ds
=\lim_{x\to 0}\int_0^t\frac {\sin s|\xi|}{\xi^2\sqrt{\xi^2+m^2}}\ddot\zeta(t-s)ds.
\end{equation}
Applying the  Fubini and the Lebesgue theorems, we get
\begin{eqnarray}\nonumber
&&\lim_{\ve\to 0+}\lim_{x\to 0}K_m^{\ve}{\cal F}^{-1}_{\xi\to x}\int_0^t\frac {\sin s|\xi|}{\xi^2\sqrt{\xi^2+m^2}}\ddot\zeta(t-s)ds
=\frac{1}{(2\pi)^3}\lim_{\ve\to 0+}\lim_{x\to 0} \int_{\R^3}e^{-i\xi\cdot x}\Big(\int_0^t\frac {\sin s|\xi|}{\xi^2(\xi^2+m^2)^{\frac 12+\ve}}\ddot\zeta(t-s)ds\Big)d^3\xi\\
\nonumber
&&=\frac{1}{2\pi^2}\lim_{\ve\to 0+}\int_0^t\Big(\int_0^t\frac {\sin sr}{(r^2+m^2)^{\frac 12+\ve}}dr\Big)\ddot\zeta(t-s)ds
=\frac{\sqrt\pi}{4\pi^2}\Gamma(\frac 12)\lim_{\ve\to 0+}\int_0^ts^{\ve}[{\rm I}_{\ve}(ms)-{\bf L}_{-\ve}(ms)]\ddot\zeta(t-s)ds\\
\label{eqq1}
&&=\frac{\sqrt\pi}{4\pi^2}\int_0^t[{\rm I}_{0}(ms)-{\bf L}_{0}(ms)]\ddot\zeta(t-s)ds
\end{eqnarray}
by  \cite [Formula  2.3.(6)]{E}.  Here ${\rm I}_{\ve}(z)$  is the modified Bessel function, and   ${\bf L}_{-\ve}(z)$ is the modified Struve function, satisfying
\begin{equation}\label{IL}
 {\rm I}_{\ve}(z)\sim (\frac 12 z)^{\ve}/\Gamma(\ve+1),\quad  {\bf L}_{-\ve}(z)\sim (\frac 12 z)^{-\ve+1}, \quad z\to 0,
\end{equation}
due to  formulas (10.30.1) and (11.2.2) of \cite{O}.
On the other hand, the  Fubini theorem implies
\begin{eqnarray}\nonumber
\!\!\!\!\!\!\!\!\!\!\!\!\!\!\!\lim_{x\to 0}{\cal F}^{-1}_{\xi\to x}\int_0^t\frac {\sin s|\xi|}{\xi^2\sqrt{\xi^2+m^2}}\ddot\zeta(t-s)ds
\!\!\!&=&\!\!\!\lim_{\rho\to 0+}\frac{1}{2\pi^2\rho}\int_0^t\Big(\int_0^\infty \frac{\sin sr\sin\rho r}{r\sqrt{r^2+m^2}}dr \Big)\ddot\zeta(t-s)ds\\
\label{eqq2}
\!\!\!&=&\!\!\!\lim_{\rho\to 0+}\frac{1}{2\pi^2}\int_0^t T_1(s,\rho)\ddot\zeta(t-s)ds
+\lim_{\rho\to 0+}\frac{1}{2\pi^2}\int_0^t T_2(s,\rho)\ddot\zeta(t-s)ds,
\end{eqnarray}
where 
$$
T_1(s,\rho)=\frac 1{\rho}\int_0^{\frac 1s}\frac{\sin sr\sin\rho r}{r\sqrt{r^2+m^2}}dr,\qquad
T_2(s,\rho)=\frac 1{\rho}\int_{\frac 1s}^\infty \frac{\sin sr\sin\rho r}{r\sqrt{r^2+m^2}}dr.
$$
Evidently,
\begin{equation}\label{eqq3}
|T_1(s,\rho)|\le \int_0^{\frac 1s}\frac{sr}{\sqrt{r^2+m^2}}dr\le 1,
\quad s>0,\quad \rho>0.
\end{equation} 
Further, we represent $T_2(s,\rho)$  as 
\[
T_2(s,\rho)=\frac{1}{2\rho}\int\limits_{1/s}^\infty \frac{\cos r(s-\rho)-\cos r(s+\rho)}{r\sqrt{r^2+m^2}}dr
=\frac{|s-\rho|}{\rho}\int\limits_{\frac{|s-\rho|}{s}}^\infty\frac{\cos u\,du}{u\sqrt{u^2+(s-\rho)^2m^2}}
-\frac{s+\rho}{\rho}\int\limits_{\frac{s+\rho}{s}}^\infty\frac{\cos u\,du}{u\sqrt{u^2+(s+\rho)^2m^2}}.
\]
 In the case $0<s\le 2\rho$, we obtain
\begin{equation}\label{eqq4}
|T_2(s,\rho)|\le \frac{|s-\rho|}{\rho} \int\limits_{\frac{|s-\rho|}{s}}^\infty\frac {du}{u^2}
+ \frac{s+\rho}{\rho} \int\limits_{\frac{s+\rho}{s}}^\infty\frac {du}{u^2}
\le \frac{x\to\xi|s-\rho|}{\rho} \frac{s}{|s-\rho|}+\frac{s+\rho}{\rho} \frac{s}{s+\rho}\le 4.
\end{equation}
In the case  $0<\rho\le s/2$, we obtain
\begin{eqnarray}\nonumber
\!\!\!\!\!\!\!\!\!\!\!\!\!\!\!|T_2(s,\rho)|\!\!\!&\le&\!\!\! \frac{s-\rho}{\rho}\int\limits_{\frac{s-\rho}{s}}^{\frac{s+\rho}{s}}\frac{du}{u\sqrt{u^2+(s-\rho)^2m^2}}
+ \frac{s+\rho}{\rho}\int\limits_{\frac{s+\rho}{s}}^{\infty}\Big(\frac{1}{u\sqrt{u^2+(s-\rho)^2m^2}}-\frac{1}{u\sqrt{u^2+(s+\rho)^2m^2}}\Big)du\\
\nonumber
\!\!\!&\le&\!\!\!\frac{s}{\rho}\int\limits_{\frac{s-\rho}{s}}^{\frac{s+\rho}{s}}\frac{du}{u^2}
+\frac{3s}{2\rho}\int\limits_{1}^{\infty}\frac{\sqrt{u^2+(s+\rho)^2m^2}-\sqrt{u^2+(s-\rho)^2m^2}}
{u(s-\rho)(s+\rho)m^2}\,du\\
\label{eqq5}
\!\!\!&\le&\!\!\!\frac{s}{\rho}\Big(\frac{s}{s-\rho}-\frac{s}{s+\rho}\Big)
+\frac{3}{s\rho}\int\limits_{1}^{\infty}\frac{\big[(s+\rho)^2-(s-\rho)^2\big]du}{u(\sqrt{u^2+(s+\rho)^2m^2}+\sqrt{u^2+(s-\rho)^2m^2})}
\le 4+\frac{6}{s\rho}\int\limits_{1}^{\infty}\frac{s\rho}{u^2}\,du\le 10.
\end{eqnarray}
Due to (\ref{eqq3})--(\ref{eqq5}), we can apply the Lebesgue theorem in (\ref{eqq2}) and obtain
$$
\lim_{x\to 0}{\cal F}^{-1}_{\xi\to x}\int_0^t\frac {\sin s|\xi|}{\xi^2\sqrt{\xi^2+m^2}}\ddot\zeta(t-s)ds
=\frac{1}{2\pi^2} \int_0^t\Big(\int_0^\infty\frac {\sin sr\, dr}{\sqrt{r^2+m^2}}\Big)\ddot\zeta(t-s)ds
=\frac{\sqrt\pi}{4\pi^2}\int_0^t[{\rm I}_{0}(ms)-{\bf L}_{0}(ms)]\ddot\zeta(t-s)ds,
$$
which coincides  with the right hand side of (\ref{eqq1}). Hence, (\ref{use1}) follows.
\end{proof}
\begin{proposition}\label{PropD} 
For any  $\zeta (t)\in C^1 [0,\infty)\otimes\C^4$, such that $\ddot\zeta\in L^2_{loc}[0,\infty)\otimes\C^4$, there exists
\begin{equation}\label{Dp}
\lim_{\ve\to 0+}\lim_{x\to 0}K_m^{\ve}D_m^{-1}p_S(x,t)
=\frac{m\beta}{4\pi}\Big(\zeta_0\Big[mt\int\limits_{mt}^{\infty}\frac{J_1(u)du}{u}-J_0(mt)\Big]
\!+\zeta(t)-m\int\limits_0^t \Big(\int\limits_{ms}^{\infty}\frac{J_1(u)du}{u}\Big)\zeta(t-s)ds\Big),\quad t> 0.
\end{equation}
\end{proposition}
Note that $D_m^{-1}=D_m D_m^{-2}=-i\alpha\cdot\nabla D_m^{-2}+m\beta D_m^{-2}$.  Then  the Proposition follows from two lemmas below.
\begin{lemma}\label{LemD} The following limit holds,
\begin{equation}\label{Dp1}
\lim_{\ve\to 0+}\lim_{x\to 0}K_m^{\ve}D_m^{-2}p_S(x,t)
=\frac{1}{4\pi}\Big(\zeta_0\Big[mt\int\limits_{mt}^{\infty}\frac{J_1(u)du}{u}-J_0(mt)\Big]
+\zeta(t)-m\int\limits_0^t \Big(\int\limits_{ms}^{\infty}\frac{J_1(u)du}{u}\Big)\zeta(t-s)ds\Big),\quad t> 0.
\end{equation} 
\end{lemma}
\begin{proof}
Note that $p_S(x,t)$ is a solution to (\ref{CP3}) with $\dot\zeta(t)$ instead of $\zeta(t)$. Hence,
\begin{equation}\label{hp0}
\hat p_S(\xi,t)=\int_0^t\frac{\sin s\sqrt{\xi^2+m^2}}{\sqrt{\xi^2+m^2}}\dot\zeta(t-s)ds,\quad
\widehat{D_m^{-2} p_S}(\xi,t)=\int_0^t\frac{\sin s\sqrt{\xi^2+m^2}}{(\xi^2+m^2)\sqrt{\xi^2+m^2}}\dot\zeta(t-s)ds,\quad t>0.
\end{equation}
Applying (\ref{use}) with $\dot\zeta$ instead of $\ddot\zeta$, we get
\begin{equation}\label{p-lim}
\lim_{\ve\to 0+}\lim_{x\to 0}K_m^{\ve}D_m^{-2}p_S(x,t) =\lim_{x\to 0}D_m^{-2}p_S(x,t).
\end{equation}
It remains to calculate $\lim\limits_{x\to 0}D_m^{-2}p_S(x,t)$. Integrating  by parts in (\ref{hp0}), we obtain
\begin{equation}\label{hp1}
\widehat{D_m^{-2} p_S}(\xi,t)=-\frac{\sin t\sqrt{\xi^2+m^2}}
{(\xi^2+m^2)\sqrt{\xi^2+m^2}}\zeta(0)+\int_0^t\frac{\cos s\sqrt{\xi^2+m^2}}{\xi^2+m^2}\zeta(t-s)ds.
\end{equation}
Let us calculate the inverse Fourier transform of 
$\frac{\cos s\sqrt{\xi^2+m^2}}{\xi^2+m^2}$ and
of $\frac{\sin t\sqrt{\xi^2+m^2}}{(\xi^2+m^2)\sqrt{\xi^2+m^2}}$.
In the sense of distributions, we obtain
\begin{eqnarray}\nonumber
{\cal F}^{-1}_{\xi\to x} \frac{\cos s\sqrt{\xi^2+m^2}}{\xi^2+m^2}
\!\!\!\!&=&\!\!\!{\cal F}^{-1}_{\xi\to x} \Big(\frac{1}{\xi^2+m^2}-\int_0^s\frac{\sin u\sqrt{\xi^2+m^2}}{\sqrt{\xi^2+m^2}}du\Big) \\
\nonumber
\!\!\!&=&\!\!\!\frac{e^{-m|x|}}{4\pi |x|}-\int_0^s\Big(\frac{\delta(u-|x|)}{4\pi |x|}
-\frac{m}{4\pi}\frac{\theta(u-|x|)J_1(m\sqrt{u^2-x^2})}{\sqrt{u^2-x^2}}\Big)du\\
\label{ek}
\!\!\!&=&\!\!\!\frac{e^{-m|x|}}{4\pi |x|}-\frac{\theta(s-|x|)}{4\pi |x|}+
\frac{m}{4\pi}\int_0^s\frac{\theta(u-|x|)J_1(m\sqrt{u^2-x^2})}{\sqrt{u^2-x^2}}du.
\end {eqnarray}
Similarly,
\begin{eqnarray}\nonumber
\!\!\!\!\!\!\!\!\!\!\!\!\!\!\!\!\!\!{\cal F}^{-1}_{\xi\to x} \frac{\sin t\sqrt{\xi^2+m^2}}{(\xi^2+m^2)\sqrt{\xi^2+m^2}}
&=&{\cal F}^{-1}_{\xi\to x} \Big(\int_0^t\frac{\cos s\sqrt{\xi^2+m^2}}{\xi^2+m^2}ds\Big)\\
\label{ek1}
\!\!\!&=&\!\!\!\int_0^t \Big(\frac{e^{-m|x|}}{4\pi |x|}-\frac{\theta(s-|x|)}{4\pi |x|}+  
\frac{m}{4\pi}\int_{0}^s\frac{\theta(u-|x|)J_1(m\sqrt{u^2-|x|^2})}{\sqrt{u^2-|x|^2}}du\Big)ds
\end {eqnarray}
Hence,   (\ref{hp1}) -(\ref{ek1}) imply for $t>0$ and $|x|\le t$
\begin{eqnarray}\nonumber
D_m^{-2}p_S(x,t)&=&-\zeta_0\Big(t\frac{e^{-m|x|}-1}{4\pi |x|}+\frac{1}{4\pi}
+\frac{m}{4\pi}\int_{|x|}^t\Big(\int_{|x|}^s\frac{J_1(m\sqrt{u^2-x^2})}{\sqrt{u^2-x^2}}du\Big)ds\Big)\\
\label{iimp}
&+&\int_0^t\Big(\frac{e^{-m|x|}}{4\pi |x|}-\frac{\theta(s-|x|)}{4\pi |x|}\Big)\zeta(t-s)ds+
\frac{m}{4\pi}\int_{|x|}^t\Big(\int_{|x|}^s\frac{J_1(m\sqrt{u^2-|x|^2})}{\sqrt{u^2-|x|^2}}du\Big)\zeta(t-s)ds.
\end {eqnarray}
Note, that
\begin{eqnarray}\nonumber
\lim_{|x|\to 0}\int_0^t\Big(\frac{e^{-m|x|}}{4\pi |x|}-\frac{\theta(s-|x|)}{4\pi |x|}\Big)\zeta(t-s)ds
&=&\lim_{|x|\to 0}\frac{e^{-m|x|}-1}{4\pi |x|}\int_0^t\zeta(t-s)ds+\lim_{|x|\to 0}\frac{1}{4\pi|x|}\int_0^{|x|}\zeta(t-s)ds\\
\nonumber
&=&-\frac{m}{4\pi}\int_0^t\zeta(s)ds +\frac{1}{4\pi}\zeta(t).
\end {eqnarray}
Moreover,
\begin{eqnarray}\nonumber
\lim_{|x|\to 0}\int_{|x|}^t\Big(\int_{|x|}^s\frac{J_1(m\sqrt{u^2-|x|^2})}{\sqrt{u^2-|x|^2}}du\Big)ds
&=&\int_0^t\Big(\int_0^s\frac{J_1(mu)}{u}du\Big)ds=t\int_0^t \frac{J_1(mu)du}{u}-\int_0^t J_1(mu)\,du\\
\nonumber
&=&t\int_0^t \frac{J_1(mu)du}{u}+\frac 1m J_0(mt)-\frac{1}{m}.
\end {eqnarray}
Substituting  this into (\ref{iimp}), we obtain
\begin{eqnarray}\nonumber
4\pi\lim_{|x|\to 0}D_m^{-2}p_S(x,t)&=&\zeta_0\Big(\!tm-tm\!\int\limits_0^t \frac{J_1(mu)du}{u}-J_0(mt)\Big)
-m\!\int\limits_0^t\zeta(s)ds +\zeta(t) +m\!\int\limits_{0}^t\Big(\!\int\limits_{0}^s\frac{J_1(mu)}{u}du\Big)\zeta(t-s)ds\\
\label{p-lim1}
&=&\zeta_0\Big(\!tm\int\limits_{mt}^{\infty} \frac{J_1(u)du}{u}-J_0(mt)\Big) +\zeta(t)
+m\!\int\limits_{0}^t\Big(\!\int\limits_{ms}^{\infty}\frac{J_1(u)}{u}du\Big)\zeta(t-s)ds,
\end {eqnarray}
since $\displaystyle\int\limits_0^{\infty}\frac{J_1(u)du}{u}=1$ by \cite[Formula 6.561(17)]{GR}.
Finally, (\ref{p-lim})  and (\ref{p-lim1})  imply (\ref{Dp1}).
\end{proof}
\begin{lemma}\label{LemD1} The following limit holds
\begin{equation}\label{Dp11}
\lim_{\ve\to 0+}\lim_{x\to 0}K_m^{\ve}\nabla D_m^{-2}p_S(x,t)=0.
\end{equation} 
\end{lemma}
\begin{proof}
Due to (\ref{hp0}),
$$
\widehat{K_m^{\ve}D_m^{-2} p_S}(\xi,t)=\int_0^t\frac{\sin s\sqrt{\xi^2+m^2}}
{(\xi^2+m^2)^{1+\ve}\sqrt{\xi^2+m^2}}\dot\zeta(t-s)ds,\quad\ve\ge 0,\quad t>0.
$$
Integrating here by parts, we obtain
\begin{equation}\label{intpart1}
\widehat{K_m^{\ve}D_m^{-2} p_S}(\xi,t)
=-\frac{\cos t\sqrt{\xi^2+m^2}}{(\xi^2+m^2)^{2+\ve}}\dot\zeta(0)+\frac{1}{(\xi^2+m^2)^{2+\ve}}\dot\zeta(t)
-\int_0^t \frac{\cos s\sqrt{\xi^2+m^2}}{(\xi^2+m^2)^{2+\ve}}\ddot\zeta(t-s)ds.
\end {equation}
We have
$$
\lim_{x\to 0} K_m^{\ve}\nabla_j{\cal F}^{-1}_{\xi\to x}\frac{1}{(\xi^2+m^2)^{2+\ve}}=\frac 1{(2\pi)^3}
\int_{\R^3}\frac{i\xi_j }{(\xi^2+m^2)^{2+\ve}}\,d^3\xi=0,\quad\ve>0,
$$
$$
\lim_{x\to 0} K_m^{\ve}\nabla_j{\cal F}^{-1}_{\xi\to x}\frac{\cos t\sqrt{\xi^2+m^2}}{(\xi^2+m^2)^{2+\ve}}dk
=\frac 1{(2\pi)^3}\int_{\R^3}\frac{i\xi_j  \cos t\sqrt{\xi^2+m^2}}{(\xi^2+m^2)^{2+\ve}}\,d^3\xi=0,\quad\ve>0.
$$
Substituting this into   (\ref{intpart1}), we  get (\ref{Dp11}). 
\end{proof}
\begin{remark}\label{rek}
Note, that  the limit (\ref{Dp11}) does not exist without the smoothing operator $K_m^\ve$.
This  immediately follows from the Taylor expansion of $D_m^{-2}p_S(x,t)$: using  (\ref{iimp}),  we obtain
$$
D_m^{-2}p_S(x,t)=C(t)-\frac{\dot\zeta(t)}{8\pi}|x|+{\cal O}(|x|^2),\quad |x|\to 0, \quad t\ge 0,
$$
where 
$$
C(t)=\frac 1{4\pi}\Big(\zeta_0\Big[tm-1-m\int_0^t\Big(\int_0^t\frac{J_1(mu)}{u}du\Big)ds\Big]+\zeta(t)-m\int_0^t\zeta(s)ds
+m \int_0^t\Big(\int_0^t\frac{J_1(mu)}{u}du\Big)\zeta(t-s)ds\Big).
$$
Evidently, $\lim\limits_{x\to 0}\nabla |x|$ does not exist.
\end{remark}
\hspace{-6mm}
\subsection {Dirac equation with  the source $\chi(t)D_m^{-2}\delta(x)=\chi(t)g(x)$}
For arbitrary  $\chi(t)\in C [0,\infty)\otimes\C^4$,   consider  the  equation 
\begin{equation}\label{DG}
i\dot{h}(x,t)  =  D_m h(x,t) +\chi(t)\,  g(x),\quad h(x,0)=h_0(x).
\end{equation}
\begin{lemma}\label{e-1}
Let $\chi(t)\in C[0,\infty)\otimes\C^4$, with  $\dot \chi\in L^1_{loc}[0,\infty)\otimes\C^4$, and let $h_0 \in H^2(\R^3)$.  Then the solution $h(x,t)$ to (\ref{DG}) satisfies
$$
h(\cdot,t)\in C([0,\infty),H^{\frac 32-}(\R^3)).
$$.
\end{lemma}
\begin{proof}
We  represent $h(x,t)$ as the sum $h(x,t)=u(x,t)+v(x,t)$,
where $u(x,t)$ is a solution to the free Dirac equation 
with  initial data $h_0(x)$, and
$v(x,t)$ is a solution to (\ref{DG}) with zero initial data. 
Evidently, $u(\cdot,t)\in C([0,\infty), H^2(\R^3))$. It remains to prove that
\begin{equation}\label{u2}
v(\cdot,t)\in C([0,\infty), H^{\frac 32-\ve}(\R^3))\quad {\rm for}~ {\rm any}\quad  \ve>0.
\end{equation}
We represent $u(x,t)$ as $u(x,t)=(-i\pa_t -D_m){\rm w}(x,t)$, where  ${\rm w}(x,t)$ is the solution to
$$
\ddot{\rm w}(x,t)  = (-\Delta+m^2){\rm w}(x,t)+\chi(t)\, g(x),\quad {\rm w}(x,0)=0,\quad \dot {\rm w}(x,0)=0.
$$
Then, for (\ref{u2}) we need to prove that
\begin{equation}\label{w-ek}
{\rm w}(\cdot,t)\in C([0,\infty), H^{\frac 52-\ve}(\R^3)),\quad \dot {\rm w}(\cdot,t)\in C([0,\infty), H^{\frac 32-\ve}(\R^3))\quad {\rm for}~ {\rm any}\quad  \ve>0.
\end{equation}
Applying the Fourier  transform, we obtain
\begin{eqnarray}\nonumber
(\xi^2+m^2)^{\frac 54-\frac{\ve}2}\,\widetilde{{\rm w}} (\xi,t)&=&
\int_0^t\frac{\sin s\sqrt{\xi^2+m^2}}{(\xi^2+m^2)^{\frac 14+\frac{\ve}2}} ~\chi(t-s)ds,\\
\nonumber
(\xi^2+m^2)^{\frac 34-\frac{\ve}2}\,\widetilde{\dot {\rm w}}(\xi,t)&=&
\int_0^t\frac{\cos s\sqrt{\xi^2+m^2}}{(\xi^2+m^2)^{\frac 14+\frac{\ve}2}} ~\chi(t-s)ds.
\end{eqnarray}
Then (\ref{w-ek}) is equivalent to
$$
\big\Vert\int\limits_0^t\frac{\sin s\sqrt{\xi^2+m^2}}{(\xi^2+m^2)^{\frac 14+\frac{\ve}2}}\chi(t-s)ds\big\Vert^2<C_1(t),\quad
\big\Vert\int\limits_0^t\frac{\cos s\sqrt{\xi^2+m^2}}{(\xi^2+m^2)^{\frac 14+\frac{\ve}2}}\chi(t-s)ds\big\Vert^2<C_2(t),\quad\ve>0, \quad t>0.
$$
Both integrals are estimated in the same way, and we consider the first integral only.
One has
$$
\big\Vert\int\limits_0^t\frac{\sin s\sqrt{\xi^2\!+m^2}}{(\xi^2\!+m^2)^{\frac 14+\frac{\ve}2}}\chi(t\!-s)ds\big\Vert^2
=\big\Vert\int\limits_0^t \chi(t\!-s)\frac{d\cos s\sqrt{\xi^2\!+m^2}}{(\xi^2\!+m^2)^{\frac 34+\frac{\ve}2}}\big\Vert^2
=4\pi \int\limits_0^\infty |q(r,t)|^2\frac{r^2dr}{(r^2\!+m^2)^{\frac 32+\ve}}\le C(t)
$$
since
$$
|q(r,t)|:=|\int\limits_0^t\,\chi(t-s) d\big(\cos s\sqrt{r^2\!+m^2}\big)|\le |\chi(0)|+|\chi(t)|+\int\limits_0^t\,|\dot\chi(s)| ds\le C(t),\quad t>0.
$$
\end{proof}
\section{Proof of well-posedness}
\label{nonlin-sect}
First, we modify the nonlinearity $F$ so that it becomes Lipschitz continuous.
Define
\begin{equation}\label{Lambda}
\Lambda(\psi_0)=\sqrt{({\cal H}_F(\psi_0)+a)/b},
\end{equation}
where $\psi_0\in {\cal D}_F$ is the initial data from Theorem \ref{theorem-well-posedness} 
and $a$, $b$ are constants from (\ref{bound-below}).
Then we may pick a modified potential function
$\tilde U(\zeta)\in C^2(\C^4,\R)$, so that\\
i) the identity holds
\begin{equation}\label{Lambda1}
\tilde U(\zeta)= U(\zeta),\quad |\zeta|\le\Lambda(\psi_0),
\end{equation}
ii) $\tilde U(\zeta)$ satisfies (\ref{bound-below}) with the same constant  $a$, $b$ as $U(\zeta)$ does:
\begin{equation}\label{Lambda2}
\tilde U(\zeta)\ge b|\zeta|^2 -a,\quad \zeta\in\C^4,
\end{equation}
iii) the functions $\tilde F_j(\zeta_j)=\pa_{\ov\zeta_j}\tilde U(\zeta)$ are Lipschitz continuous:
\begin{equation}\label{Lambda22}
|\tilde F_j(\zeta_j)-\tilde F_j(\eta_j)|\le C|\zeta_j-\eta_j|,\quad\zeta_j,\eta_j\in\C.
\end{equation}
We suppose that $\psi_0=f+\zeta_0 g$, where $f\in H^2(\R^3)$,
and  consider the Cauchy problem for (\ref{Dirac})) with the modified nonlinearity $\tilde F$.
As before we denote by $\psi_f(x,t)\in C([0,\infty), H^2(\R^3))$  the unique solution to  (\ref{CP1}), and by $\psi_{\zeta_0g}(x,t)\in C([0,\infty), L^2(\R^3))$
the  unique solution to  (\ref{CP2}) with $e=\zeta_0$. Let 
$\lambda(t)$ and $\mu(t)$  are  defined by (\ref{lam}) and  by (\ref{mu}).
The following lemma is proved by standard argument from the contraction mapping principle.
\begin{lemma}\label{LLWP}
Let conditions  (\ref {Lambda1})--(\ref {Lambda22}) be satisfied. 
Then there exists $\tau>0$ such that  the Cauchy problem
\begin{eqnarray}\nonumber
&&\lambda(t)+\zeta_0\mu(t)+\frac{1}{4\pi}\Big(m\zeta(t)-\dot\zeta(t)-m\int\limits_0^t\frac{J_1(ms)}{s}\zeta(t-s)ds\Big)\\
\label{delay}
&&+\frac{im\beta}{4\pi}\Big(\zeta_0\Big[tm\int\limits_{mt}^{\infty}\frac{J_1(u)du}{u}-J_0(mt)\Big]
+\zeta(t)-m\int\limits_0^t \Big(\int\limits_{ms}^{\infty}\frac{J_1(u)du}{u}\Big)\zeta(t-s)ds\Big)=\tilde F(\zeta(t)),\quad \zeta(0)=\zeta_{0}
\end{eqnarray}
has  unique solution $\zeta\in C^1[0,\tau]\otimes\C^4$.
\end{lemma}
Denote 
\begin{equation}\label{pp-sol}
\psi_S(x,t):=\varphi_S(x,t)+iD_m^{-1}\zeta_0\dot\gamma(x,t)+ iD_m^{-1}p_S(x,t),
\end{equation}
where $\gamma(x,t)$ is defined in (\ref{gamma}), and
$$
\varphi_S(x,t)= \frac{\theta(t-|x|)}{4\pi|x|}\zeta(t-|x|)-\frac{m}{4\pi}
\int_0^t\frac{\theta(s-|x|)J_1(m\sqrt{s^2-|x|^2})}{\sqrt{s^2-|x|^2}}\zeta(t-s)ds,
$$
$$
p_S(x,t)= \frac{\theta(t-|x|)}{4\pi|x|}\dot\zeta(t-|x|)-\frac{m}{4\pi}
\int_0^t\frac{\theta(s-|x|)J_1(m\sqrt{s^2-|x|^2})}{\sqrt{s^2-|x|^2}}\dot\zeta(t-s)ds
$$ 
with $\zeta(t)$ from Lemma \ref{LLWP}.
Now we establish the local well-posedeness for (\ref{Dirac}).
\begin{proposition}\label{TLWP}(Local well-posedeness).
Let  the conditions  (\ref {Lambda1})--(\ref {Lambda22}) hold.
Then the function
\begin{equation}\label{sol-sum}
\psi(x,t):= \psi_f(x,t)+\psi_{\zeta_0g}(x,t)+\psi_S(x,t)\in D_{\tilde F}, \quad t\in [0,\tau]
\end{equation}
is a unique  solution to the Cauchy problem  
\begin{equation}\label{CP}
\left\{\begin{array}{l}
i\dot{\psi}(x,t) = D_m\psi(x,t)-D_m^{-1}\zeta(t) \delta(x)\\\\
\lim\limits_{\ve \to 0+}\lim\limits_{x \to 0}\, K_m^{\ve}\left( \psi(x,t)-\zeta(t) g(x)\right)=\tilde F(\zeta(t))\\\\
\psi(x,0) = \psi_0(x)=f(x)+\zeta_0g(x)
\end{array}\right.
\end{equation}
\end{proposition}
\begin{proof}
Using  (\ref{lam}),   (\ref{rw}),   (\ref{pp-sol}), (\ref{mu}),  (\ref{dif1})  and (\ref{Dp}) successively, we get 
\begin{eqnarray}\nonumber
&&\lim_{\ve \to 0+}\lim_{x \to 0}\, K_m^{\ve} \left( \psi(x,t)\!-\!\zeta(t) g(x)\right)\\
\nonumber
&&=\lim_{\ve \to 0+}\lim_{x \to 0}K_m^{\ve} \left( \psi_{f}(x,t)+ \psi_{\zeta_0g}(x,t)+\psi_S(x,t)-\zeta(t) g(x)\right)\\
\nonumber
&&=\lambda(t)+\lim_{\ve \to 0+}\lim_{x \to 0}\, K_m^{\ve}\left(\zeta_0 g(x)-\zeta_0\gamma(x,t)-iD_m^{-1}\zeta_0\dot\gamma(x,t)
+ \varphi_S(x,t)+iD_m^{-1}\zeta_0\dot\gamma(x,t)+ iD_m^{-1}p_S(x,t)-\zeta(t) g(x)\right)\\
\nonumber
&&=\lambda(t)+\zeta_0\mu(t)+\lim_{\ve \to 0+}\lim_{x \to 0}K_m^{\ve}\left( \varphi_S(x,t)-\zeta(t) g(x))+ iD_m^{-1}p_S(x,t)\right) \\
\nonumber
&&=\lambda(t)+\zeta_0\mu(t)+\frac{1}{4\pi}\big(m\zeta(t)-\dot\zeta(t)-m\int_0^t\frac{J_1(ms)}{s}\zeta(t-s)ds\big)\\
\nonumber
&&+\frac{im\beta}{4\pi}\Big(\zeta_0\Big[tm\int\limits_{mt}^{\infty}\frac{J_1(u)du}{u}-J_0(mt)\Big]
+\zeta(t)-m\int_0^t \Big(\int\limits_{ms}^{\infty}\frac{J_1(u)du}{u}\Big)\zeta(t-s)ds\Big)=\tilde F(\zeta(t)),\quad t\ge0
\end{eqnarray}
since $\zeta(t)$ solves (\ref{delay}). Hence, the second equation of (\ref{CP}) is satisfied.  
Further,
\begin{eqnarray}\nonumber
i\dot\psi(x,t)&=&i\dot\psi_f(x,t)+i\dot\psi_{\zeta_0g}(x,t)+i\dot\psi_S(x,t)=D_m\psi_f(x,t)+D_m\psi_{\zeta_0g}(x,t)
+D_m\psi_S(x,t)-D_m^{-1}\zeta(t)\delta(x)\\
\nonumber
&=&D_m\psi(x,t)-D_m^{-1}\zeta(t)\delta(x).
\end{eqnarray}
Hence, $\psi$ solves the first equation of (\ref{CP}). 
Finally, the function
$\psi_{reg}(x,t)=\psi(x,t)-\zeta(t) g(x)$ satisfies
\begin{equation}\label{psipsi}
\psi_{reg}(\cdot,t)\in C([0,\tau], H^{\frac 32-}(\R^3)).
\end{equation}
Indeed,  $\psi_{reg}(x,t)$ is a solution to 
\begin{equation}\label{psi-reg}
i\dot\psi_{reg}(x,t)=D_m\psi_{reg}(x,t)-i\dot\zeta(t) g(x)
\end{equation}
with  initial data $f\in H^2(\R^3)$ and with $\dot\zeta\in C[0,\infty)\otimes\C^4$  satisfying  equation (\ref{delay}). 
Lemma \ref{imp} implies that  $\dot\lambda\in L^1_{loc}[0,\infty)\otimes\C^4$. Moreover,  $\dot\mu\in C([0,\infty))$ by (\ref{mu}).
Hence,  (\ref{delay}) implies  that $\ddot\zeta\in L^1_{loc}[0,\infty)\otimes\C^4$, and  (\ref{psipsi}) holds by  Lemma \ref{e-1}.
\smallskip\\
Suppose now that $\tilde\psi(x,t)=\tilde\psi_{reg}(x,t)+\tilde\zeta(t) g(x)$  is another  solution to (\ref{CP}). 
Then, by reversing the above argument, the second equation of (\ref{CP})
implies that $\tilde\zeta(t)$ solves the Cauchy problem (\ref{delay}). 
The uniqueness of the solution of (\ref{delay}) implies that
$\tilde\zeta(t)=\zeta(t)$. Then, defining 
$$
\varphi_S(x,t)= \frac{\theta(t-|x|)}{4\pi|x|}\zeta(t-|x|)-\frac{m}{4\pi}
\int_0^t\frac{\theta(s-|x|)J_1(m\sqrt{s^2-|x|^2})}{\sqrt{s^2-|x|^2}}\zeta(t-s)ds,
$$
$$
p_S(x,t)= \frac{\theta(t-|x|)}{4\pi|x|}\dot\zeta(t-|x|)-\frac{m}{4\pi}
\int_0^t\frac{\theta(s-|x|)J_1(m\sqrt{s^2-|x|^2})}{\sqrt{s^2-|x|^2}}\dot\zeta(t-s)ds
$$
and 
$$
\psi_S(x,t):=\varphi_S(x,t)+iD_m^{-1}\zeta_0\dot\gamma(x,t)+ iD_m^{-1}p(x,t),
$$
for $\tilde\psi_{\rm free}=\tilde\psi_f+\tilde\psi_{\zeta_0g}=\tilde\psi-\psi_S$ one obtains
$$
i\dot{\tilde\psi}_{\rm free}(x,t)=D_m\tilde\psi(x,t)-D_m^{-1}\zeta(t)\delta(x)-D_m\psi_S(x,t)+D_m^{-1}\zeta(t)\delta(x)
=D_m{\tilde\psi}_{\rm free}(x,t).
$$
Thus, $\tilde\psi_{\rm free}$ solves the Cauchy problem for the free Dirac equation with initial data $f+\zeta_0g$.
Hence, by the uniqueness of the solution  to this  Cauchy problem, we have $\tilde\psi_{\rm free}=\psi_{\rm free}$, and then $\tilde\psi=\psi$.
\end{proof}
\begin{lemma}\label{H-pr}
Let conditions (\ref{Lambda1})--(\ref{Lambda22}) hold, and let 
$\psi(t)\in {\cal D}_{\tilde F}$, $t\in [0,\tau]$, be a solution to (\ref{CP}). 
Then 
\begin{equation}\label{cHFT}
{\cal H}_{\tilde F}(\psi(\cdot,t))=\Vert D_m\psi_{reg}(\cdot,t)\Vert^2+\tilde U(\zeta(t))=\const,\quad t\in [0,\tau].
\end{equation}
\end{lemma}
\begin{proof}
Equation (\ref{psi-reg}) and the second equation of (\ref{CP}) imply for any $t\in [0,\tau]$
\begin{eqnarray}\nonumber
\!\!\!\!\!\!\!\!\!\!\!\!\!\!\!\lim\limits_{\ve\to 0}\frac{d}{dt}\Vert K_m^{\ve}D_m\psi_{reg}\Vert^2
\!\!\!&=&\!\!\!\lim\limits_{\ve\to 0}\Big[\langle K_m^{\ve}D_m\dot\psi_{reg},\, K_m^{\ve}D_m\psi_{reg}\rangle 
+\langle K_m^{\ve}D_m\psi_{reg}, \,K_m^{\ve}D_m\dot\psi_{reg}\rangle\Big]\\
\nonumber
\!\!\!&=&\!\!\!\lim\limits_{\ve\to 0}\Big[\langle -iK_m^{\ve}D_m^2\psi_{reg}-K_m^{\ve}D_m\dot \zeta g,\,K_m^{\ve}D_m\psi_{reg}\rangle
+\langle K_m^{\ve}D_m\psi_{reg}, \,-iK_m^{\ve}D_m^2\psi_{reg}-K_m^{\ve}D_m\dot \zeta g\rangle\Big]\\
\nonumber
\!\!\!&=&\!\!\!\lim\limits_{\ve\to 0}\Big[-\langle D_m^{2}\dot \zeta g,\,K_m^{2\ve}\psi_{reg}\rangle-\langle K_m^{2\ve}\psi_{reg},\, D_m^{2}\dot \zeta g\rangle\Big]\\
\label{e-2}
\!\!\!&=&\!\!\!\lim\limits_{\ve\to 0}\Big[\!-\dot \zeta\!\cdot\langle  \delta(x),\,K_m^{2\ve}\psi_{reg}\rangle
-{\overline{\dot \zeta}}\!\cdot\langle K_m^{2\ve}\psi_{reg}, \,\delta(x)\rangle\Big]
\!=-\dot\zeta\! \cdot\overline {\tilde F}(\zeta)-{\overline{\dot\zeta}}\!\cdot{\tilde F}(\zeta)=-\frac{d}{dt}\tilde U(\zeta).
\end{eqnarray}
Here the scalar product  $\langle K_m^{\ve}D_m^2\psi_{reg},K_m^{\ve}D_m\psi_{reg}\rangle$ exists since
$K_m^{\ve}\psi_{reg}(\cdot,t)\in C([0,\infty),H^{3/2}(\R^3))$ for any $\ve>0$ due to Lemma \ref{e-1}.
The right hand side of (\ref{e-2}) is continuous bounded function since $\tilde U\in C^2(\C^4)$ and $\dot\zeta\in C[0,\tau]\otimes\C^4$.
Hence, in the sense of distributions$$
\frac{d}{dt}\Vert D_m\psi_{reg}(\cdot,t)\Vert^2=\frac{d}{dt}\lim\limits_{\ve\to 0}\Vert K_m^{\ve}D_m\psi_{reg}(\cdot,t)\Vert^2
=\lim\limits_{\ve\to 0}\frac{d}{dt}\Vert K_m^{\ve}D_m\psi_{reg}(\cdot,t)\Vert^2=-\frac{d}{dt}\tilde U(\zeta).
$$
Then (\ref{cHFT}) follows.
\end{proof}
\begin{corollary}\label{cor1}
The following identity holds
\begin{equation}\label{UtU}
\tilde U(\zeta(t))=U(\zeta(t)), \quad t\in [0,\tau].
\end{equation}
\end{corollary}
\begin{proof}
First note that 
\[
{\cal H}_{F}(\psi_0)\ge  U(\zeta_{0})\ge b|\zeta_{0}|^2-a.
\]
Therefore, $|\zeta_0|\le\Lambda(\psi_0)$ and then $\tilde U(\zeta_0)=U(\zeta_0)$, 
${\cal H}_{\tilde F}(\psi_0)={\cal H}_{F}(\psi_0)$.
Further, 
\[
{\cal H}_{\tilde F}(\psi(t))\ge \tilde U(\zeta(t))\ge b|\zeta(t)|^2-a,\quad t\in [0,\tau].
\]
Hence, (\ref{cHFT}) implies that
\begin{equation}\label{zeta_bound}
|\zeta(t)|\le \sqrt{({\cal H}_{\tilde F}(\psi(t))+a)/b}=\sqrt{({\cal H}_{\tilde F}(\psi_0)+a)/b}
=\sqrt{({\cal H}_{F}(\psi_0)+a)/b}=\Lambda(\psi_0),\quad t\in [0,\tau].
\end{equation}
\end{proof}
Identity  (\ref{UtU}) implies that we can replace $\tilde F$ by $F$ 
in  Proposition \ref{TLWP} and in Lemma \ref{H-pr}.
\smallskip\\
{\bf Proof of Theorem \ref{theorem-well-posedness}}. 
The solution $\psi(t)\in {\cal D}_{F}$ constructed in Proposition  \ref{TLWP}
exists for $0\le t\le\tau$, where the time span $\tau$ in Lemma \ref{LLWP} depends only on $\Lambda(\psi_0)$.
Hence, the bound (\ref{zeta_bound}) at $t=\tau$ allows us to extend the solution $\psi$ to the time
interval $[\tau, 2\tau]$. We proceed by induction to obtain the solution for all $t\ge 0$.

\end{document}